\documentclass[12pt,letterpaper,titlepage]{amsart}
\usepackage{amsmath, amssymb, amsthm, amsfonts,amscd,amsaddr}
\usepackage[all]{xy}
\usepackage[usenames]{color}

\theoremstyle{plain}
\newtheorem{theorem}{Theorem}[section]
\newtheorem{proposition}[theorem]{Proposition}
\newtheorem{cor}[theorem]{Corollary}

\newtheorem{prop}[theorem]{Proposition}
\newtheorem{lemma}[theorem]{Lemma}
\newtheorem{definition}[theorem]{Definition}
\theoremstyle{definition}
\newtheorem{ex}[theorem]{Example}
\newtheorem{rmk}[theorem]{Remark}
\numberwithin{equation}{section}
\newtheorem*{theoremA*}{Theorem A}
\newtheorem*{theoremB*}{Theorem B}
\newtheorem*{theoremm1*}{Theorem A'}
\newtheorem*{theoremC*}{Theorem C}
\newtheorem*{theoremD*}{Theorem D}
\newtheorem*{theoremE*}{Theorem E}
\newtheorem*{theoremF*}{Theorem F}
\newtheorem*{theoremE2*}{Theorem E2}
\newtheorem*{theoremE3*}{Theorem E3}

\newcommand{\bs}{\backslash}
\newcommand{\cc}{\mathcal{C}}
\newcommand{\C}{\mathbb{C}}

\newcommand{\Hc}{\mathcal{H}}

\newcommand{\Pb}{\mathbb{P}}

\newcommand{\Z}{\mathbb{Z}}

\newcommand{\R}{\mathbb{R}}
\newcommand{\N}{\mathbb{N}}

\newcommand{\SO}{\operatorname{SO}}

\newcommand{\Gr}{\operatorname{Gr}}

\newcommand{\tr}{\operatorname{tr}}
\newcommand{\im}{\operatorname{im}}

\newcommand{\Ad}{\operatorname{Ad}}
\newcommand{\ad}{\operatorname{ad}}
\newcommand{\diag}{\operatorname{diag}}
\newcommand{\vol}{\operatorname{vol}}

\newcommand{\Span}{\operatorname{span}}

\def\hat{\widehat}
\def\af{\mathfrak{a}}
\def\bmf{\mathfrak{b}}
\def\e{\epsilon}
\def\gf{\mathfrak{g}}

\def\df{\mathfrak{d}}

\def\hf{\mathfrak{h}}
\def\kf{\mathfrak{k}}
\def\lf{\mathfrak{l}}
\def\mf{\mathfrak{m}}
\def\nf{\mathfrak{n}}

\def\pf{\mathfrak{p}}

\def\uf{\mathfrak{u}}

\def\zf{\mathfrak{z}}

\def\1{{\bf1}}
\def\cC{\mathcal{C}}
\def\U{\mathcal{U}}

\def\oline{\overline}

\def\propertyUI{{\rm (I)}}

\title[Volume growth]
{Volume growth, temperedness and integrability of matrix coefficients on a real spherical space}

\begin{document}
\date{January 6, 2016}

\begin{abstract} We apply the local structure theorem from \cite{KKS} and the polar decomposition of 
\cite{KKSS} to a real spherical space $Z=G/H$ and  control the volume growth on $Z$. 
We define the Harish-Chandra Schwartz space on $Z$. We give a geometric 
criterion to ensure $L^p$-integrability of matrix coefficients on $Z$. \end{abstract}

\author[Knop]{Friedrich Knop}
\email{friedrich.knop@fau.de}
\address{Department Mathematik, Emmy-Noether-Zentrum\\
FAU Erlangen-N\"urnberg, Cauerstr. 11, 91058 Erlangen, Germany} 
\author[Kr\"otz]{Bernhard Kr\"{o}tz}
\email{bkroetz@gmx.de}
\address{Universit\"at Paderborn, Institut f\"ur Mathematik\\Warburger Stra\ss e 100,
33098 Paderborn, Germany}
\thanks{The second author was supported by ERC Advanced Investigators Grant HARG 268105}
\author[Sayag]{Eitan Sayag}
\email{eitan.sayag@gmail.com}
\address{Department of Mathematics, Ben Gurion University of the Negev\\P.O.B. 653, Be'er Sheva 84105, 
Israel}
\author[Schlichtkrull]{Henrik Schlichtkrull}
\email{schlicht@math.ku.dk}
\address{University of Copenhagen, Department of Mathematics\\Universitetsparken 5, 
DK-2100 Copenhagen \O, Denmark}

\maketitle

\section{Introduction}

Consider a real algebraic homogeneous space $Z=G/H$, attached to 
an algebraic real reductive group $G$ and assumed to carry a $G$-invariant measure. 
For a compact symmetric neighborhood  $B$ of $\bf 1$ in $G$ we
define the volume weight:
$$ {\bf v}(z):= \vol_Z(Bz)\qquad (z\in Z) \, .$$
The choice of the ball $B$ is not important, as different balls yield equivalent
weights (mutual ratios are bounded above and below by positive constants). 
Our first aim is then to study the {\it volume growth} of $Z$, the growth of ${\bf v}$ as a function of $z\in Z$.

\par In order to efficiently control the volume growth we need a sufficiently explicit 
parametrization of $Z$. This is possible in case $Z$ is {\it real spherical}, that is, 
minimal parabolic subgroups $P<G$ admit open orbits on $Z$. We assume this, and choose
$P$ such that its orbit through the standard base point
$z_0:=H$ is open.

It was shown in \cite{KKSS} that then there is a split torus $A_Z<P$, 
a closed simplicial cone $A_Z^-\subset A_Z$ (called the compression cone), a compact set $\Omega\subset G$ and a finite set 
$F\subset G$ such that $Z$ admits the generalized polar decomposition
\begin{equation}\label{pol}  
Z = \Omega A_Z^- F \cdot z_0\, . \end{equation}
Since the compact set $\Omega$ can be incorporated in $B$,
it is then sufficient to find bounds for the function  $a \mapsto {\bf v}(Baf\cdot z_0)$
on the compression cone $ A_Z^-$, for each $f\in F$.  Applying the local structure 
theorem from \cite{KKS} we obtain in Proposition~\ref{volume}
the precise asymptotic behavior of ${\bf v}$, namely
\begin{equation*} {\bf v} (af\cdot z_0) \asymp  a^{-2\rho_\uf} \qquad (a\in A_Z^-, f \in F).\end{equation*}
for a suitably defined exponent $\rho_\uf$ on $\af_Z$. 

\par The polar decomposition (\ref{pol}) leads to 
a notion of {\it temperedness} for unimodular real spherical spaces. This is accomplished
with a function space $\cC(Z)$ defined as follows.
Using (\ref{pol}) another weight 
function on $Z$ is
defined by 
$$ {\bf w}(z) :=\sup \|\log a\|,\quad (z\in Z) ,$$
where the supremum is taken over all $a\in A_Z^-$
such that $z\in\Omega a F\cdot z_0$. 
We then consider the  family of semi-norms
\begin{equation*} p_n (f)  := \Big(\int_Z  |f(z)|^2 ( 1 +  {\bf w}(z))^n \ dz\Big)^{\frac 12}  
\qquad (n\in \N)\,\end{equation*}
on $C_c(Z)$, with which we define a Fr\'echet completion $E$.
Here $G$ acts by the regular representation. Now $\cC(Z)$ is defined
as the set $E^\infty$ of smooth vectors. 
We refer to it as the {\it Harish-Chandra Schwartz space} on $Z$,
thus generalizing a notion  for symmetric spaces from
\cite{vdB}.

In Proposition \ref{HCS space} we give a different characterization of $\cC(Z)$
by means of the volume weight $\bf v$. For $f\in C_c(Z)$ let 
\begin{equation*}   
q_n (f) := \sup_{z\in Z} |f(z)| \sqrt{{\bf v}(z)}  
(1 + {\bf w}(z))^n 
\qquad (n\in \N)\,.\end{equation*}
Using a result from \cite{B}
we show that $\cC(Z)$ coincides with
the space of smooth vectors for the Fr\'echet completion
obtained from this family of seminorms.
The general results of \cite{B} 
imply further that $\mathcal{C}(Z)$ is nuclear. 

\par Finally we investigate $L^p$-integrability of generalized 
matrix coefficients. Here we are given a unitary representation $(\pi, \Hc)$ of $G$ and
an $H$-invariant distribution vector $\eta\in \Hc^{-\infty}:=(\Hc^\infty)'$. Every smooth vector 
$v\in \Hc^\infty$ then gives rise to a smooth function on $Z$, the generalized 
matrix-coefficient associated to $v$ and $\eta$: 
$$ m_{v, \eta}( g\cdot z_0) = \eta(\pi(g^{-1})v) \qquad (g\in G)\, .$$

We denote by $H_\eta\supset H$ the full stabilizer of $\eta$, and note that 
the matrix coefficient factorizes to a function on $G/H_\eta $, which we denote again by
$ m_{v, \eta}$. 
It can happen that $H_\eta$ is strictly larger than $H$: take for example $\pi$ the trivial 
representation. Less pathological examples occur for the triple spaces $G/H = G_0 \times G_0 \times G_0/ \diag(G_0)$
with $G_0$ a Lorentzian group. These are real spherical, and it happens for some pairs $(\pi,\eta)$
that $H\subsetneq H_\eta\subsetneq G$.

\par We say that $Z$ has {\it property (I)} provided that for all unitary irreducible 
representations $(\pi,\Hc)$ and all $\eta\in (\Hc^{-\infty})^H$
the following holds: 
\begin{enumerate}
\item $Z_\eta:= G/H_\eta$ carries a $G$-invariant measure. 
\item There exists $1\leq p<\infty$ such that for all $v\in \Hc^\infty$ the matrix coefficient 
$m_{v, \eta}$  
belongs to $L^p(Z_\eta)$. 
\end{enumerate}

According to \cite{KKSS} there always exists a maximal split torus $A$ in $P$
such that
$A_Z \subset A$ and
\begin{equation} \label{wf} 
A_Z^-\cdot z_0 \supset A^- \cdot z_0, \end{equation}
with $A^-\subset A $ the closure of the negative chamber
determined by the unipotent radical $N$ of $P$. If this is accomplished 
with equality in (\ref{wf}), then $Z$ is called 
{\it wavefront}.  Symmetric spaces, for instance, are wavefront.  
The main result on $L^p$-integrability is then: 

\begin{theorem} Let $Z=G/H$ be a wavefront real spherical space with $H$ self-normalizing and reductive.  Then $Z$ has property (I). 
\end{theorem} 

\par This result will be applied to a 
lattice counting problem in a subsequent 
publication \cite{KSS1}.

\section{Homogeneous real spherical spaces}\label{Section 2}

Real Lie groups in this paper will be denoted by upper case Latin letters, $A$, $B$ etc, and 
their associated Lie algebras with the corresponding lower case Gothic letter $\af$, $\bmf$ etc.
The identity component of a real Lie group $A$ is denoted by $A_0$. 

\par Throughout we assume that $G$ is an algebraic real reductive group, i.e.  
there is a connected complex reductive group $G_\C$ with Lie algebra
$\gf_\C=\gf\otimes_\R \C$ such that $G\subset G_\C$.

\par Let $H<G$ be a subgroup with finitely many components and form the 
homogeneous space $Z=G/H$. 
We assume that $Z$ is {\it real algebraic}, i.e.~there is complex 
algebraic subgroup $H_\C <G_\C$ such that $G\cap H_\C =H$. 
Note that $Z_\C:=G_\C/H_\C$ carries naturally the structure of a complex $G_\C$-variety for which we 
have a $G$-equivariant embedding 
$$Z\hookrightarrow Z_\C, \ \ gH\mapsto gH_\C\, .$$
Observe that $Z$ is a union of connected components (with respect to the Euclidean topology) 
of the real points $Z_\C(\R)$ of $Z_\C$.

\par With the letter $P$ we denote a minimal parabolic subgroup $P<G$. 
We call $Z$ {\it real spherical} provided that  
$P$  admits an  open orbit on $Z$. We assume this, and that $P$ is chosen so that
$PH$ is open in $G$. 
\par For a reductive Lie algebra $\gf$ we write $\gf_{\mathrm n}$ for the direct sum of the non-compact 
ideals in $[\gf,\gf]$. If $\gf$ is the Lie algebra of $G$, then $G_{\mathrm n}$ denotes the
corresponding connected normal subgroup of $[G,G]$.

We recall the local structure theorem of \cite{KKS}: There exists a unique parabolic subgroup 
$Q\supset P$ with Levi-decomposition $Q= LU $ such that: 
\begin{itemize} 
\item $QH=PH$. 
\item $L_{\mathrm n} \subset Q\cap H = L\cap H $. 
\item The map
\begin{equation}\label{LST} Q \times_L (L/L\cap H) \to P\cdot z_0, \ \ [q,l (L\cap H)]\mapsto ql H 
\end{equation}
is a $Q$-equivariant diffeomorphism.
\end{itemize}

Note that $D:=L/L_{\rm n}$ is a Lie group with
compact Lie algebra $\df$. Observe that 
$\df:=\zf(\lf)+\lf_{\rm c}$ where $\zf(\lf)$ is the center of $\lf$
and $\lf_{\rm c}$ 
the direct sum of all compact simple ideals in $\lf$. 
Hence $L/L\cap H$ is a homogeneous space for the group $D$, say  $D/C$, 
and it follows from (\ref{LST}) that the map 
\begin{equation}\label{eq2} 
U \times D/C\to P\cdot z_0\end{equation}
is a diffeomorphism. 
As $D$ is algebraic it decomposes as  a product of subgroups
\begin{equation}\label{decomp D}
 D=A_D \times M_D
\end{equation}
with $A_D\simeq 
(\R^+)^n$ a connected real torus and $M_D$ a compact group (observe that $A_D$ is the connected component of an algebraic torus). As $C$ is an 
algebraic subgroup of $D$ we have $C=A_C\times M_C$ with $A_C=A_D\cap C$ 
and $M_C=M_D\cap C$.
We set $A_Z:=A_D/A_C$ and $M_Z=M_D/M_C$ 
and record 
$$D/C \simeq A_Z \times M_Z$$
Then $\dim\af_Z$ is an invariant of $Z$ which we call its {\it real rank}, see \cite{KKS}.

Let $L=K_LA_L N_L$ be an Iwasawa decomposition of $L$, and let
$G=KAN$ be an Iwasawa decomposition of $G$ which is compatible,
that is, $K_L=K\cap L$, $A=A_L$ and $N_L=N\cap L$.  We denote by $M$ the centralizer of 
$\af$ in $K$, and by
$\theta$ an algebraic Cartan involution of $G$ for which
$K$ is the set of $\theta$-fixed points.

Write $\Sigma^+\subset \af^*$
for the set of positive roots attached to $\nf$ and set 
$$\af^{-}=\{X \in \af\mid (\forall \alpha \in \Sigma^+)\  \alpha(X)\leq 0\}$$
for the closure of the negative Weyl chamber.

\par Set $A_H:=A\cap H$, $M_H = M \cap H$ and note that 
$A_Z \simeq A/A_H$. Note that $A_Z=\exp(\af_Z)$ for
$\af_Z=\af/\af\cap\hf$, and that $A_H$ is connected
as $A$ has no torsion elements.
In the introduction we realized $A_Z$ as a subgroup of $A$ with Lie algebra 
$\af_Z\subset \af$. As that requires the choice of a complement,
the realization as a quotient is preferable.

\subsection{The compression cone and the polar decomposition}

Let $Z=G/H$ be a homogeneous real spherical space. 
We recall here a few results of \cite{KKSS}, Sect 5, 
about the compression (or valuation) cone $\af_Z^-\subset \af_Z$ of $Z$.

\par Note that $P_\C H_\C \subset G_\C$ is a Zariski open which is affine 
by the local structure theorem. We denote by ${\mathcal P}_{++}$ the monoid of regular 
functions on $G_\C$, which are real valued on $G$, and  whose zero set is 
$G_\C \bs P_\C H_\C$. Attached to 
$f \in {\mathcal P}_{++}$ are algebraic characters $\psi_f, \chi_f$ of $(H_\C)_0$ and $P_\C$   
such that 
$$f(ph) =\chi_f(p) \psi_f(h) f({\bf 1})\qquad (h \in (H_\C)_0, p \in P_\C)\, .$$
Every $f\in {\mathcal P}_{++}$ gives rise to a finite dimensional irreducible 
real representation $V_f:=\Span_\R \{ R(G)f\} \subset \R[G]$  with $R(g)f:= f(\cdot g)$. 
We write $v_H:=f \in V_f$ and note that $[v_H]\in \Pb(V_f)$  
can be chosen such that it is $H$-fixed (see \cite{KKSS}, Lemma 3.11(b)).
Let $V_f^*$ be the dual of $V_f$ and let $v_0^*\in V_f^*$ be the evaluation $h\mapsto h(\1)$,
then $[v_0^*]\in \Pb(V_f^*)$ is $P$-fixed, hence a highest weight vector.
Let $v_0\in V_f$ be a lowest weight vector, normalized such that $v_0^*(v_0)=1$. 
Assuming also $f(\1)=1$ we obtain
\begin{equation}\label{m-coeff}  
f(g) = v_0^*( g\cdot v_H) \qquad (g\in G)
\end{equation} 
and in particular for $a \in A$
\begin{equation*}
a\cdot v_0 = \chi_f(a) v_0, \quad
a\cdot v_0^* = \chi_f(a)^{-1} v_0^*\,. 
\end{equation*}
Hence $\lambda_f:= d\chi_f({\bf 1})\in \af^*$ is the lowest weight of $V_f$.

Define %$\af_f^{--}$ to be the set of those $X\in \af$ for which 
\begin{equation*} %\label{limit} 
\af_f^{--}:=\{ X\in\af_Z\mid \lim_{t\to \infty}  [\exp(tX)v_H]=[v_0]\}\,\end{equation*}
and note that this is well defined as $[v_H]$ is $H$-fixed.
It follows from \cite{KKSS} 
that $\af_f^{--}$ is independent of the choice of 
$f \in {\mathcal P}_{++}$, hence we denote it $\af_Z^{--}$. 
The closure of $\af_Z^{--}$ is denoted by $\af_Z^-$ and we refer 
to it as the {\it compression cone} of $Z$.  Set $A_Z^-:=\exp(\af_Z^-)<A_\C/ A_\C\cap H_\C$.

\begin{rmk}\label{dcone} 
(See \cite{KKSS}, Sect. 5) The cone $\af_Z^-$ is finitely generated.  For $f\in {\mathcal P}_{++}$ 
note that $[v_H]\in \Pb(V_f)$ was chosen to be
an $H$-fixed element. Expand  
$$ v_H:= \sum_{\mu \in \af^*}  v_\mu $$ 
into $\af$-weight spaces.  
Let $\Lambda_f\subset \af^*$ be the set of $\nu\in \af^*$ for which 
$v_{\lambda_f +\nu}\neq 0$.  Note that $\Lambda_f$ is naturally a subset of 
$\af_Z^*=\af_H^\perp\subset \af^*$. 
Then for $X\in \af_Z$ it is immediate from the definition that 
\begin{equation} \label{Lambdanu} 
X \in \af_f^{--} \iff \nu(X) <0 \quad \hbox{for all $\nu \in \Lambda_f\setminus\{0\}$}\, .\end{equation}
\end{rmk}

Write $\Sigma_\uf \subset \Sigma^+$ for the roots with root 
spaces in $\uf$.  For a root $\alpha\in \Sigma$ we denote 
by $\alpha^\vee\in \af$ the corresponding co-root.
Observe that $\Lambda_f \subset \N_0[\Sigma_\uf]$ and that 
$\lambda_f(\alpha^\vee) <0$ for all $\alpha\in \Sigma_\uf$ 
for $f \in {\mathcal P}_{++}$.  

\par In this paper we use a refinement of the polar decomposition of $Z$ obtained in \cite{KKSS}.
We recall from \cite{KKSS} the group 
$$J:=\{ g\in G\mid P_\C H_\C g = P_\C H_\C\}$$
which has the properties that $N_G(H)<J$ and $J/N_G(H)$ is compact. Further it was shown that there 
exists an irreducible finite dimensional representation $(\pi, V)$ of $G$ with $J$-fixed vector $v_J$ such that 
$$ Z_J:=G/J \to {\mathbb P}(V), \ \ gJ \mapsto [\pi(g)v_J]$$ 
defines an embedding.  The closure  $\oline Z_J$ of $Z_J=G/J$ in ${\mathbb P}(V)$ was referred to as 
simple compactification of $Z_J$ in \cite{KKSS}.  The decisive property was that $\oline Z_J $ featured a unique 
closed $G$-orbit.  Passing from $Z_J$ to $Z$ is technically a bit cumbersome as the component group 
of $J$ is not explicitely  known.  Instead of using $J$ we prefer here to use a more manageable subgroup $H_1<G$
which features almost the same properties as $J$. 
In order to define $H_1$ let $\af_{Z,E}=\af_Z^-\cap (-\af_{Z}^-)$ be the 
edge of $\af_Z^-$ which we realize as a subspace of $\af$ via the Cartan-Killing form, i.e. $\af_Z\simeq \af_H^\perp \subset \af$. 
Note that $\af_{Z,E}$ normalizes $\hf$ and therefore $\hf_1:=\hf +\af_{Z,E}$ defines a subalgebra. 
We let $H_{1,\C}$ be the connected algebraic subgroup of $G_\C$ with Lie algebra $\hf_{1,\C}$ and set $H_1:=G\cap H_{1,\C}$.
Observe that $H_1<N_G(H)$ and  $N_G(H)/H_1$  is compact. Moreover with $A_{Z,E}=\exp(\af_{Z,E})$ we note that 
hat $H_1/HA_{Z,E}$ is a finite group.  Let $F_0\subset H_1$ be a set of representatives of $H_1/HA_{Z,E}$.

Set $Z_1:=G/H_1$ and observe that  $\af_{Z_1}=\af_Z/\af_{Z,E}$ and
$\af_{Z_1}^-=\af_Z^-/\af_{Z,E}$. 
Then, according to \cite{KK} Sect. 7, there exists 
a finite dimensional (not necessarily irreducible) representation $\pi: G\to V$ with $H_1$-spherical 
vector $v_{H_1}$ such that $Z_1\to {\mathbb P}(V), \ gH_1\mapsto [\pi(g)v_{H_1}] $ embeds into projective space. Moreover, 
the closure $\oline Z_1$ has a unique closed $G$-orbit.  This is all what is needed 
to derive in analogy to \cite{KKSS},  Section 5,  the {\it polar decomposition} 
of $Z_1$: 
\begin{equation} \label{POLARa} 
Z_1 = \Omega A_{Z_1}^- F_1 \cdot z_{0,1} \end{equation}
where $\Omega\subset G$ is a compact set of the form
\begin{equation} \label{POLAR1} 
\Omega = F' K \end{equation}
for a finite set $F'$, and $F_1\subset G$ is a finite set such that 
\begin{equation} \label{POLAR2a} 
F_1\cdot z_{0,1} \subset \exp(i\af_{Z_1}) \cdot z_{0,1}\subset Z_{1,\C} \, .\end{equation}
Let us mention that $F_1$ is a set of representatives for the open $P$-orbits on $Z_1$. 

\par It is now a simple matter to lift the polar decomposition from $Z_1$ to $Z$ (see \cite{KKS2}, Section 3.4 and in particular Lemma 3.4). 
We obtain
\begin{equation} \label{POLAR} 
Z = \Omega A_{Z}^- F \cdot z_0 \end{equation}
where $\Omega\subset G$ is as above 
and $F\subset G$ is a finite set such that 
\begin{equation} \label{POLAR2} 
F\cdot z_0 \subset T_Z F_0\cdot z_0\subset Z_\C \end{equation}
with $T_Z:=\exp(i \af_H^\perp)<\exp(i\af)$. 

Note that (\ref{POLAR2}) implies that $Px\cdot z_0$ is open in 
$Z$ for each $x\in F$. Indeed, if $g\in T_Z  N_G(H)$ 
then 
\begin{equation}\label{F-spherical}
\pf_\C+\Ad(g)\hf_\C=\gf_\C\,.
\end{equation}

\subsection{The limiting subalgebra}

In order to discuss volume growth we need 
another property of the compression 
cone. Let $\oline \uf=\theta(\uf)$ and define
\begin{equation*}%\label{hlim}
\hf_{\rm lim} = \oline \uf + \hf\cap\lf\,,
\end{equation*}
then $\hf_{\rm lim}$ is a spherical subalgebra with
$d=\dim \hf = \dim \hf_{\rm lim}$. 
This follows from the decompositions
$$\gf=\hf\oplus \af_Z\oplus \mf_Z\oplus \uf=\oline\uf\oplus (\hf\cap\lf)\oplus \af_z\oplus \mf_Z\oplus \uf.$$
By the same decompositions we find a linear map $T:\oline\uf \to \pf$ 
such that
\begin{equation}\label{Tgraph}
 W+Y\mapsto W+Y+T(Y),\qquad W\in \lf\cap\hf, Y\in\oline\uf
\end{equation}
provides
a linear isomorphism
$\hf_{\rm lim} \to  \hf\, .$

\begin{rmk}
Let $d=\dim\hf$. We write 
$\Gr_d(\gf)$ for the Grassmannian of $d$-dimen\-sional 
subspaces of the real vector space $\gf$ and
recall from \cite{KKSS}, Sect.~5, the following characterization of 
the compression cone. 

%\begin{lemma} \label{Grass} 
Let $X\in \af_Z$. The following are 
equivalent: 
\begin{enumerate}
\item $X\in \af_Z^{--}$. 
\item $\lim_ {t\to \infty} e^{t\ad X} \hf = 
\hf_{\rm lim}$ in $\Gr_d(\gf)$. 
\end{enumerate}
%\end{lemma}
\end{rmk}

\subsection{Quasiaffine real spherical spaces}

A real spherical space is called quasi-affine if there is a non-zero $G$-equivariant 
rational map $Z \to V$ for a rational real $G$-module $V$. 
If we denote by ${\mathcal P}_{++, \bf 1}$ the subset of ${\mathcal P}_{++}$ which corresponds to right 
$H_\C$-invariant functions, then 
quasi-affine means that ${\mathcal P}_{++, {\bf 1}}\neq \emptyset$. 

\par Note that it is not such a severe restriction to assume that $Z$ is quasiaffine:
for $f\in {\mathcal P}_{++}$ with characters $(\chi_f, \psi_f)$ one obtains 
a quasiaffine space $Z_1 =G_1/ H_1$ with 
$G_1 = G \times \R^\times$ and $H_1:=\{ (h, \psi_f^{-1}(h))\mid h \in H\}$.  
Observe that $Z_1$ is a $G$-space and that there is a natural $G$-fibering 
$$ \R^\times \to Z_1 \to Z\, .  $$
Furthermore we have the following relation between compression cones:
\begin{equation}\label{quasiaffine cone} 
\af_{Z_1}^-= \af_Z^-\oplus \R\, .\end{equation}

\section{Weight functions}

%\par In this section we assume that $Z$ is a real spherical space.  

Recall the notion of weight function from \cite{B}: A positive 
function $w$ on $Z$ is called a {\it weight} provided for all compact sets $B$ there 
exists a constant $C_B\ge 1$ such that 
$$ w(g\cdot z) \leq C_B w(z) \qquad (z\in Z, g\in B)\, .$$

Two positive functions $w_1, w_2$ on $Z$ are called {\it equivalent} if the 
quotient ${w_1}/ {w_2}$ is bounded from above and below by positive constants. 

For later reference, we note that if $w$ is a weight then so is $\frac1w$,
and likewise the function
\begin{equation}\label{wtilde}
\tilde w:=\max\{ \log w, c-\log w\},
\end{equation}
when $c>0$ is a constant.

%\begin{rmk}\label{averaged matrix coeffs are weight fcts}
When $Z$ is quasi-affine one can construct weight functions on it from finite dimensional 
representations. 

\begin{lemma}\label{rep weight}
%{averaged matrix coeffs are weight fcts}
Let $Z$ be quasi-affine, and let $(\pi,V)$ be a finite dimensional irreducible representation with a non-zero
$H$-fixed vector $v_H$. 
\begin{enumerate}
\item The function 
$w(z)=\|g\cdot v_H\|$ for $z=g\cdot z_0\in Z$ is a weight.
\item Assume $V=V_f$ for $f\in{\mathcal P}_{++, {\bf 1}}$
and let $w_f=w$ as in {\rm (1)}.
Let $B\subset G$ be compact.
Then there exist $C_1,C_2>0$ such that 
$$C_1\chi_f(a)\le w_f(bax\cdot z_0) \le C_2\chi_f(a)$$
for all $b\in B$, $a\in A_Z^-$, and $x\in F$. 
\end{enumerate}
\end{lemma}

\begin{proof} (1) is just because $\pi$ is bounded.
For (2) we first observe that (1) implies
$$C_1 w(a)= C_1 w(b^{-1}ba) \le
w(ba)\le C_2 w(a).$$
Hence we may assume $B=\{1\}$.
Next we observe that by Remark \ref{dcone}
$$ a \cdot v_H = \chi_f(a) \left(v_{\lambda_f} + \sum_{\nu\in\Lambda_f\bs \{0\}} a^\nu v_{\lambda_f+\nu}\right),$$
which is an orthogonal sum. 

Finally we recall from (\ref{POLAR2})
that elements $x$ of the finite set $F$ can be written as $x=x_1 x_2$ with $x_1=t_1 h_1$ 
with $t_1\in T_Z$, $h_1\in  H_\C$,
and  $x_2\in N_G(H)$.
Observe that $h_1x_2\cdot v_H = c_0v_H$ for a constant $c_0$. 
Hence 
$$ ax \cdot v_H = c_0 \chi_f(at_1) \left(v_{\lambda_f} + \sum_{\nu\in\Lambda_f\bs \{0\}} (at_1)^\nu v_{\lambda_f+\nu}\right),$$
from which it follows that
$$|c_0|\,\chi_f(a) \|v_{\lambda_f}\|\le 
\|ax \cdot v_H \| \le |c_0|\, \chi_f(a) \|v_H\|$$
for $a\in A_Z^-$. Here the second inequality is
obtained with (\ref{Lambdanu}).
\end{proof}

\begin{prop}\label{uniform A-ball}
Let $B_1,B_2\subset G$ be compact sets. 
There exists a 
compact set $B_A \subset A_Z$ such that 
$a_1a_2^{-1}\in B_A$ for all pairs $a_1,a_2\in A_Z^-$
with $B_1a_1x_1\cdot z_0 \cap B_2a_2x_2\cdot z_0 \neq \emptyset$
for some $x_1,x_2\in F$.
\end{prop}

\begin{proof} It follows from Lemma \ref{rep weight} that
for each $f\in {\mathcal P}_{++, {\bf 1}}$
there exist  $C_1,C_2>0$ such that
$$C_1 \chi_f(a_i) \le w(z) \le C_2 \chi_f(a_i), \quad i=1,2,$$
for all $z\in B_1a_1x_1\cdot z_0 \cap B_2a_2x_2\cdot z_0$.
Hence  if this set is non-empty then $\chi_f(a_1a_2^{-1})$ belongs to a compact 
neighborhood of $1$.

We recall from \cite{KKS} Remark 3.5,
that the semigroup generated by the $\lambda_f$ 
has rank $\dim \af_Z$. The proposition follows.
\end{proof}

We are mainly interested in the compact sets $B\subset G$ which satisfy
the polar decomposition
\begin{equation} \label{B-POLAR} 
Z = B A_Z^- F\cdot z_0\, .
\end{equation}
and recall from (\ref{POLAR}) that this is the case when $B=\Omega$.
By a {\it ball} in a Lie group 
we understand a compact symmetric neighborhood of 
the neutral element. 
The following is then an immediate consequence of 
Proposition \ref{uniform A-ball}.

\begin{cor}\label{uniform A-ball cor}
For every compact set $B\subset G$ there exists a ball
$B_A$ in $A_Z$ such that  $BaF\cdot z_0 \subset \Omega(A_Z^-\cap B_A a)F\cdot z_0$
for all $a\in A_Z^-$.
\end{cor}

We fix a norm $\|\cdot\|$ on $\af_Z$. 

\begin{prop}\label{w is a weight} 
Let $B\subset G$ be a compact set for which the polar decomposition
{\rm (\ref{B-POLAR})}
is valid and define
$$ \mathbf{w} (z)=\mathbf{w}_B(z) := \sup_{\{a\in A_Z^-\mid 
z \in B a F \cdot z_0 \}} 
 \|\log a\|\,,\quad (z\in Z).$$
\begin{enumerate}
 \item The function $W_B(z)=e^{\mathbf{w}_B(z)}$ is a weight on $Z$, and all weights obtained in this 
fashion are equivalent.
\item There exists a constant $C\in\R$ such that
$$\|\log a\|\leq \mathbf{w}(\omega ax\cdot z_0)\le C+\|\log a\|$$
for all $\omega\in B$, $a\in A_Z^-$, $x\in F$.
 \item If $B$ is sufficiently large then $\mathbf{w}$ is a weight.
\end{enumerate}
\end{prop}

\begin{proof} 
(1) Note that (\ref{B-POLAR})
ensures that the supremum is taken over a non-empty 
subset of $A_Z^-$.
Moreover, as this set is compact, the supremum is finite.

With (\ref{quasiaffine cone}) the proof is easily reduced to the case where $Z$ is 
quasi-affine. 
It then follows from 
Proposition \ref{uniform A-ball}
that if $B_1,B_2$ are compact sets both satisfying
(\ref{B-POLAR}) then there exists  $C>0$ such that
\begin{equation}\label{W_B equivalence}
W_{B_2}(z)\leq CW_{B_1}(z)
\end{equation}
for all $z\in Z$.

It easily follows from (\ref{W_B equivalence}) and the identity
$$W_B(g\cdot z)=W_{g^{-1}B}(z), \qquad g\in G, z\in Z.$$
that $W_B$ is a weight. 
The equivalence also follows from (\ref{W_B equivalence}).

(2)
The first inequality is clear. 
For the second we apply Proposition \ref{uniform A-ball}
with $B_1=B_2=B$. Let $z=\omega a x\cdot z_0$
and let $\omega'\in B$, $a'\in A_Z^-$ and $x'\in F$
be such that $z=\omega'a'x'\cdot z_0$ 
and $\mathbf{w}(z)=\|\log a'\|$.
Then $a'a^{-1}\in B_A$ and hence $\|\log a'\|\le
C+ \|\log a\|$ for some $C>0$ depending
only on the size of $B_A$.

(3) If $W$ is a weight on $Z$ and $\inf_{z\in Z} W(z)>1$, then $\log W$ is also a weight
(see (\ref{wtilde})).
Hence we only need to prove $\inf_{z\in Z} \mathbf{w}_B(z)>0$.
We assume that $B$ contains $\Omega B_A$ where $B_A$ is a ball in $A$.
We obtain
for each $z\in Z$ that $z=kax\cdot z_0$ for some $k\in \Omega$, $a\in A_Z^-$, and $x\in F$.
Hence it follows from the identity $z=k b^{-1}bax\cdot z_0$ that 
$$\mathbf{w}(z)\ge \sup_{\{b\in B_A\mid ba\in A_Z^-\}}\|\log (ba)\|.$$ 
This implies the uniform positive lower bound for $\mathbf{w}$.
\end{proof}

\section{Volume growth}
In this section and the following
we assume that $Z$ is a unimodular real spherical space. 
We are interested in effective upper and lower bounds of 
\begin{equation*}%\label{defi wB}
\mathbf{v}_B(g)= \vol_Z(Bg\cdot z_0) 
\end{equation*}
in dependence of $g \in G$, for a fixed ball $B$. 

\begin{lemma}\label{v is a weight}
The function $\mathbf{v}_B$ is a weight and its equivalence class is independent of the ball $B$. 
\end{lemma}

\begin{proof} See \cite{B}, Lemma-Definition 3.3 and its proof.
\end{proof}

\par In the sequel we  drop the index $B$ and write $\mathbf{v}$ instead of $\mathbf{v}_B$.  
Effective bounds for ${\bf v}$ allow a characterization of the  
tempered spectrum of $L^2(Z)$ as we know from \cite{B} and explicate in the next section
when we define the Harish-Chandra Schwartz space $\mathcal{C}(Z)$ of $Z$.  

\par We shall give lower and upper bounds on  $\vol_Z(Ba\cdot z_0)$ for $a$
in $A_Z^-$.  
As before we let  $Q\supset P$ 
so that 
there is an isomorphism (\ref{eq2})
$$ U \times M_Z \times A_Z \simeq P\cdot z_0\, .$$
In the sequel we view $M_Z\subset P\cdot z_0$.
We define  
$\rho_\uf\in \af^*$ by $\rho_\uf (X) = \frac12 \tr_\uf \ad X $
and note that

\begin{lemma}\label{rhou}
 $\rho_\uf=0$ on $\af_H$.
\end{lemma}
 \begin{proof} Follows, since $\gf/\hf$ and $\lf/\af\cap\hf$ are both 
unimodular, and $\gf/\hf$ is equivalent to $\uf+\lf/\lf\cap\hf$ as an $\af_H$-module.
\end{proof}

Denote by $\mu_Z$ a Haar-measure on $Z$. 
It follows that for a continuous function 
$f$ on $Z=G/H$,  compactly supported in $P\cdot z_0$,  
\begin{equation}\label{int}
\int_Z f(z)  d\mu_Z(z) = \int_U \int_{M_Z}\int_{A_Z} f ( u m a \cdot  z_0) a^{-2\rho_\uf}  da \ dm \ du
\end{equation}
where $du$, $dm $ and $da$ are suitably normalized 
Haar measures on $U$, $M_Z$ and $A_Z$ respectively.  
The integrand makes sense for $a\in A_Z$ because of Lemma \ref{rhou}.

\par As the choice of the ball $B$  does not matter, 
we may assume that it is invariant under multiplication by $L\cap K$ from the right
(for example, we can replace $B$ by $KBK$). It then follows, see (\ref{decomp D}),
that $BaM_Z = Ba\cdot z_0$ for all $a\in A_Z$. 
Then $$(B \cap  U)(B\cap A_Z )a M_Z \subset B^2 a\cdot z_0$$
which, in view of (\ref{int})  yields a constant $C>0$ such that 
$$ C\cdot a^{-2\rho_\uf}  \leq \vol_Z (Ba\cdot z_0) \qquad (a\in A_Z)\, . $$

An upper bound is obtained on the negative chamber as follows.

\begin{prop} \label{volume} For every ball $B$ in $G$ there exist constants
$C_1, C_2>0$ such that 
\begin{equation}\label{bound}  
C_1 \cdot a^{-2\rho_\uf} \leq {\bf v}(ax\cdot z_0) \leq C_2 \cdot a^{-2\rho_\uf}\qquad 
(a\in A_Z^-, x\in F) \end{equation}
\end{prop}

Before giving the proof we note the following result.

\begin{lemma} \label{volume lemma}
Let $U\subset G$ and $D\subset H$  be open sets.
 Then $$\vol_Z(U\cdot z_0)\vol_H(D)\le \vol_G(UD).$$
\end{lemma}

\begin{proof} We exhaust $U\cdot z_0$ by a countable union of disjoint 
sets of the form  $g_i\exp(V_i)\cdot z_0$ 
where each $V_i\subset \gf$ belongs to a fixed slice for the exponential map $\gf\to G/H$,
in a neighborhood of the origin, and where $g_i\exp(V_i)\subset U$. 
Then $UD$ contains the disjoint union of the sets
$g_i\exp(V_i)D$, and for each of these we find
$$\vol_Z(g_i\exp(V_i)\cdot z_0)\vol_H(D)= \vol_G(g_i\exp(V_i)D).\qedhere$$
\end{proof}

\begin{proof} [Proof of Proposition \ref{volume}]
We first assume that $x=\1$. The lower bound is already established 
so that we can focus on the upper bound. 

Let $\Xi\subset \hf$ be a compact symmetric neighborhood of $0\in \hf$ such that 
$\Xi \to H, X \mapsto \exp(X)$ is diffeomorphic onto its image. 
It follows from Lemma \ref{volume lemma} that there is a constant $C_0>0$ such that 
$$
\vol_Z (B a \cdot z_0)\leq C_0 \vol_G (B a \exp(\Xi))
$$

Let $B_A$ be a ball in $A_Z$. In view of (\ref{POLAR}), (\ref{POLAR1}),
and Corollary \ref{uniform A-ball cor},
it is then sufficient 
to show that 
$\vol_G (K B_A a \exp(\Xi))\leq C 
a^{-2\rho_\uf} $ for all $a \in A_Z^-$. 
Now 
$$
\vol_G (KB_A a \exp(\Xi))=\vol_G \Phi_a(K\times B_A\times \Xi)\, 
$$
where
$$ \Phi_a: K \times B_A \times \Xi \to G, \qquad  (k,b,X)\mapsto kb\exp(\Ad(a)X)\, .$$
\par 
We wish to find a uniform bound for the differential of 
$ \Phi_a$ at all $(k,b,X)\in K\times B_A\times\Xi$ and for 
all $a\in A_Z^-$.
For that we recall the isomorphism (\ref{Tgraph}).
Let $X=W+Y+T(Y)\in\hf$ where $W\in \lf\cap\hf$, $Y\in\oline\uf$, then
\begin{equation}\label{Ad(a)X}
\Ad(a)X=W+\Ad(a)(Y+T(Y)).
\end{equation}
\par Observe further that $\kf +\af +\hf_{\rm lim}=\gf$. Let 
$W_1, \ldots, W_k$ be a basis of $\lf \cap \hf$, $Y_1, \ldots, Y_m$ 
be a basis of $\oline \uf $, $V_1, \ldots, V_r$ a basis of 
$\af_Z$. Finally let $U_1, \ldots, U_s$ be independent elements 
from $\kf$, such that 
$$W_1, \ldots, W_k, Y_1, \ldots, Y_m, V_1, \ldots, V_r , U_1, \ldots, U_s$$
is a basis of $\gf$. 
To simplify notation we write ${\bf U}=U_1\wedge
\ldots \wedge U_s$ etc. 
For ${\bf Y}\in \bigwedge^k\oline \uf$ we use the notation 
${\bf Y}+ T({\bf Y})$ for 
$$ (I+T)(\bf Y)=(Y_1 + T(Y_1) )\wedge \ldots \wedge (Y_m +T(Y_m))
\in\bigwedge^k\hf\, . $$
Then 
\begin{equation}\label{limit}
\Ad(a) ( {\bf Y} +  T({\bf Y}) )=a^{-2\rho_\uf} ({\bf Y} +R_Y(a))
\end{equation}
with $R_Y(a)$ bounded on $A_Z^-$. 

\par We obtain a volume form $\omega$ on $\gf$ by 
$$\omega= {\bf W}\wedge {\bf Y}\wedge{\bf V}\wedge{\bf U}\, .$$
Note that by (\ref{Ad(a)X}) the determinant of $ d \Phi_a(\1,\1, 0))$ 
is given by 
$$ (\det d\Phi_a(\1, \1, 0))\cdot  \omega=   {\bf W}\wedge 
\left[\Ad (a) ({\bf Y}+ T({\bf Y}))\right]\wedge{\bf V}\wedge{\bf U}\, .$$
Hence by (\ref{limit}) the Jacobian of $\Phi_a$ at $(\1,  \1, 0)$  is bounded  
by $ C_2a^{-2\rho_\uf}$
with $C_2>0$ a constant which is independent of $a\in A_Z^-$.
As the formulas show, the bound is not changed by small distortions: 
for $\det d\Phi_a(k, b, \xi)$ with $k\in K$, $b\in B_A$ and $\xi\in \Xi$ the 
same bound holds true. This proves the proposition for $x=\1$. 

Recall that elements $x\in F$ decompose as $x=x_1 x_2 $ with $x_2\in N_G(H)$
and  $x_1\in \exp(i\af_Z) H_\C \cap G$. A quick inspection of the 
proof above yields  the same bounds with arbitrary $x\in F$. 
\end{proof}

\section{Harish-Chandra Schwartz space on $Z$}

Let $\Gamma_Z <A_Z$ be a lattice, that is $\log \Gamma_Z\subset \af_Z $ is a lattice in the vector 
space $\af_Z$. Set $\Gamma_Z^-:=\Gamma_Z \cap A_Z^-$. 
After enlarging $B$ if necessary we can assume that 
\begin{equation*} %\label{net} 
Z = B \Gamma_Z^- F \cdot z_0\, .\end{equation*}
With this we have the domination
\begin{equation}\label{integral as sum}
\int_Z f(z)\,dz\, \le \sum_{\gamma\in \Gamma_Z^-, x\in F}
\int_{B\gamma x\cdot z_0}f(z)\,dz
\end{equation}
for every non-negative measurable function on $Z$.
Note also that 
\begin{equation} \label{summation} 
\sum_{\gamma\in \Gamma_Z^- } ( 1+ \|\log \gamma\|)^{-s}<\infty
\end{equation}
for $s > \dim \af_Z$.

For $u\in \U(\gf)$ and a smooth function $f$ on $Z$ we write $L_u f$ for the left derivative of $f$ with respect 
to $u$.  With consider two families of semi-norms on $C_c^\infty (Z)$, 
$$ p_{n,u} (f) :=   \| ( 1+ {\bf w} )^n L_uf \|_2 \qquad (u \in \U(\gf), n\in \N) $$ 
and 
$$ q_{n,u} (f) := \sup_{z\in Z}  (1 + {\bf w}(z))^n \sqrt{{\bf v}(z)} |L_uf(z) |\qquad (u\in \U(\gf), n\in \N) \, .$$

\begin{proposition} \label{HCS space}
The two families $(p_{n,u})$ and $(q_{n,u})$ define the same 
locally convex topology on $C_c^\infty(Z)$.
\end{proposition}

\begin{proof} We apply (\ref{integral as sum}) to the definition of $p_{n,u}(f)$. 
It follows from  Lemma \ref{v is a weight} that
$\int_{B y\cdot z_0} \mathbf v(z)^{-1}\, dz$
is a bounded function of $y\in Z$. Using
(\ref{summation}) we then obtain 
$p_{n,u}(f)\leq C q_{m,u}(f)$ when $s=2m-2n>\dim\af_Z$.
\par The invariant Sobolev lemma (see \cite{B}, ``key lemma'' in section 3.4), 
provides a domination in the opposite direction. \end{proof}

The completion of $C_c^\infty(Z)$ with respect to either family 
is denoted ${\mathcal C}(Z)$ and called
the {\it Harish-Chandra Schwartz space} of $Z$. 
The space $\cC(Z)$ was introduced by Harish-Chandra 
for $Z=G=G\times G/G$ (see \cite{W}, Sect. 7.1, for 
a simplified exposition), and it was extended to symmetric spaces by van den Ban in \cite{vdB}, Sect. 17. 

By the local Sobolev lemma we see that  
$\cC(Z)\subset C^\infty(Z)$.

\begin{proposition} \label{prop nuclear} The inclusion 
$$ {\mathcal C}(Z) \hookrightarrow L^2(Z)$$
is fine, i.e. there exists a continuous Hermitian norm $p$ on ${\mathcal C}(Z)$ such that the completion ${\mathcal C}(Z)_p$ of 
$({\mathcal C}(Z), p)$ gives rise to a Hilbert-Schmidt embedding ${\mathcal C}(Z)_p \to L^2(Z)$. 
Moreover ${\mathcal C}(Z)$ is nuclear. 
\end{proposition}

\begin{proof} It follows from (\ref{summation}) that the weight 
$$z\mapsto (1 + {\bf w}(z))^n$$
is summable in the sense of \cite{B}, Sect.~3.2, provided that $n>\dim \af_Z$. 
Hence  \cite{B}, Theorem 3.2, applies to $\cc(Z)$ and the assertions follow. 
\end{proof} 

\subsection{$Z$-tempered Harish-Chandra modules}

\par Let $V$ be a Harish-Chandra module for $(\gf, K)$ and $V^\infty$ its unique 
smooth moderate growth Fr\'echet  globalization, see \cite{BK}. 
We denote  by $V^{-\infty}$ the strong dual of $V^\infty$, and by
$(V^{-\infty})^H$ its subspace of $H$-fixed vectors which we
recall is finite dimensional 
(see \cite{KO} or \cite{KS}).% for effective bounds on $\dim (V^{-\infty})^H$.

Attached to $\eta\in (V^{-\infty})^H$ and $v\in V^\infty$ is the matrix coefficient 
\begin{equation*}%\label{matrix coeff}
m_{v,\eta}(gH):= \eta(\pi(g^{-1})v) \qquad (g\in G)
\end{equation*}
which is a smooth function on $Z$.

\begin{definition} Let $Z=G/H$ be a unimodular real spherical space, 
$V$ a Harish-Chandra module for $(\gf, K)$ and $\eta\in (V^{-\infty})^H$. 
The pair $(V, \eta)$ is called {\rm $Z$-tempered}  provided that there 
exists an $n \in \Z$ such that for all $v \in V$ one has 
\begin{equation}\label{tempered bound}
\sup_{z\in Z} |m_{v, \eta}(z)| \sqrt{{\bf v}(z)} (1 + {\bf w}(z))^n  <
\infty\, .
\end{equation}
\end{definition}

\begin{rmk} Note that by
Proposition \ref{w is a weight}(2) and Proposition \ref{volume}
the bound (\ref{tempered bound}) can equivalently be expressed as follows.

For all $v\in V$ there exists a constant $C_v>0$ such that 
\begin{equation} \label{tempinq} | m_{v, \eta} (\omega a x \cdot z_0)|\leq C_v a^{\rho_Q}  (1 +\|\log a\|)^n \end{equation}
for all $\omega\in \Omega$, $x \in F$ 
and $a\in A_Z^-$. 
\end{rmk}

We recall the abstract Plancherel theorem for $L^2(Z)$: 
\begin{equation}\label{Plancherel}  
L^2(Z) \simeq \int_{\hat G} \,{\mathcal M}_\pi \otimes {\mathcal H}_\pi   \ d\mu(\pi)\, . \end{equation} 
Here ${\mathcal M}_\pi\subset (V_\pi^{-\infty})^H$ is a subspace, the multiplicity space. 
The measure class of $\mu$ is unique. Specific Plancherel measures
result from choices of the inner product on the finite dimensional 
multiplicity spaces ${\mathcal M}_\pi$.  

\par The isomorphism (\ref{Plancherel}) is given by the abstract 
inverse-Fourier-trans\-form 
which assigns to a smooth section $\hat G\ni \pi \mapsto \eta_\pi \otimes v_\pi \in 
{\mathcal M}_\pi \otimes {\mathcal H}_\pi^\infty$ the (generalized) function
on $Z$  
$$ z\mapsto \int_{\hat G} m_{v_\pi , \eta_\pi}(z) \ d\mu(\pi)\qquad (z\in Z)\, .$$

\par Let us denote by $V_\pi$ the Harish-Chandra module of the unitary 
representation $(\pi, {\mathcal H}_\pi)$.  

Proposition \ref{HCS space} combined with \cite{B} (see \cite{CD} for a nice summary of the results in \cite{B}) then yields 
the following characterization of the spectrum of $L^2(Z)$. 

\begin{prop} Let $Z$ be a unimodular real spherical space with Plancherel 
measure $\mu$. For $\mu$-almost every $\pi\in \hat G$ and each 
$\eta_\pi \in 
{\mathcal M}_\pi$ the pair $(V_\pi, \eta_\pi)$ is $Z$-tempered.
\end{prop}

\begin{rmk}
The multiplicity space ${\mathcal M}_\pi$ can be a proper subspace of
$(V_\pi^{-\infty})^H$. This happens for example for the Lorentzian
symmetric space $G/H=\SO_0(n,1)/\SO_0(n-1,1)$ when $n\ge 4$, in which
case it can be shown with methods from \cite{S}
that there exists an irreducible Harish-Chandra 
module $V$ which embeds into $L^2(G/H)$ with multiplicity one, but 
with $\dim(V_\pi^{-\infty})^H=2$.
\end{rmk}

\section{Bounds for generalized  matrix coefficients}\label{section bounds}

In this section we assume that $H$ is a closed subgroup of $G$
with algebraic Lie algebra and with finitely many components.

Let $V$ be a Harish-Chandra module for $(\gf, K)$. 
We say that a norm on $q$ on $V^\infty$ is $G$-continuous, provided that the completion 
of $(V^\infty, q)$ is a Banach representation of $G$. For a $G$-continuous norm on $V^\infty$ 
we denote by $q^*$ its dual norm. 
Note that for all $\eta\in V^{-\infty}$ there 
exists a $G$-continuous norm $p$ such that $p^*(\eta)<\infty$.

\begin{rmk}\label{global bound} 
Note that for any compact subset $\Omega\subset G$ and $v\in V^\infty$ we have $\sup_{g\in \Omega} 
q(\pi(g)v)<\infty$ for all $G$-continuous norms on $V^\infty$. 
\end{rmk}

In \cite{W} 4.3.5 or \cite{KSS} equation (3.5) one associates to $V$ an exponent
$\Lambda_V\in \af^*$, the definition of which we recall
and adapt to the current set-up.
Consider the finite-dimensional $A$-module 
$V/\oline\nf V$ 
and its
spectrum of weights, say $\mu_1, \ldots, \mu_k\in\af_\C^*$.  
Let $H_1, \ldots, H_n\in\af$ be the basis elements
which are dual to the simple roots 
for $\nf$. Then
\begin{equation*}%\label{defi Lambda}
\Lambda_V (H_i):= \max_{1\leq j\leq k} {\rm Re }\, \mu_j(H_i) \qquad 
(1\leq i \leq n).
\end{equation*}
Further we attach the integer $d_V \in \N_0$ as in 
\cite{KSS}. 

\begin{rmk}  \label{negative}
Later we shall use the following property of the exponent
$\Lambda_V$, which is a consequence of the Howe-Moore theorem (see
\cite{RS}, p. 447).
If $\gf$ is simple and $\pi$ is non-trivial
unitary, then $$\Lambda_V\in  {\rm Int} (\af^+)^*,$$
the interior of the dual cone of $\af^+$.
\end{rmk}

\begin{theorem}\label{upper bound}
Let $Z=G/H$ be as above, and
suppose for the minimal parabolic subgroup $P\subset G$
that $PH$ is open in $G$. Let $V$ be a Harish-Chandra module. Fix a $G$-continuous 
norm $p$ on $V^\infty$.  
Then there exists a $G$-continuous norm $q$ on $V^\infty$ 
such that
\begin{equation} \label{fundbound} 
|m_{v,\eta}(a\cdot z_0)| \le q(v)  
p^*(\eta) a^{\Lambda_V}
(1+\|\log a\|)^{d_V} \end{equation}
for $a \in A^-$, $v\in V^\infty$ and
$\eta\in (V^{-\infty})^H$. 
\end{theorem}

\begin{rmk}\label{upperbound with x}
Note that by applying the theorem to the subgroup $H^x=xHx^{-1}$ and
the $H^x$-fixed distribution vector $\eta^x=x.\eta$, 
one obtains for each $x\in G$ such that $PxH$ is open a similar domination 
of $m_{v,\eta}(ax\cdot z_0)$ with a norm $q$ that depends on $x$. 
\end{rmk}

\begin{rmk} In (\ref{fundbound}) the exponent $\Lambda_V$ can 
be replaced by one with 
a more refined definition that relates to $Z$.
Recall from Section \ref{Section 2} the parabolic subgroup
$Q=LU\supset P$.
If in the definition of $\Lambda_V$ one replaces the weights of 
$V/\oline\nf V$ by those of the smaller space  
$V/(\lf\cap\hf+\oline{\uf}) V$, 
one obtains an `$H$-spherical' exponent $\Lambda_{H,V}\in\af^*$. With the technique of 
\cite{KS} Thm.~3.2 the bound can then be improved with
$\Lambda_{H,V}$ in place of $\Lambda_V$. As this is not currently needed 
the details are omitted.
\end{rmk}

\begin{proof} The proof is a development of the proof of \cite{KSS}, Th.~3.2. 
In order to compare with \cite{KSS}, it is 
convenient to rewrite the statement above. We shall assume
\begin{equation}\label{open}
\bar PH \text{ open}
\end{equation}
and instead of (\ref{fundbound}) prove for all $a\in A^+$ and $v, \eta$ 
as before that
\begin{equation} \label{fundbound2}
|m_{v,\eta}(a\cdot z_0)| \le q(v)  p^*(\eta) a^{\Lambda_{V}}
(1+\|\log a\|)^{d_V} .\end{equation}

In \cite{KSS} it is assumed that $PH$ is open and that $H$ is reductive.
However, the first step of the given proof consists of the observation that then 
$\bar PH$ is also open. With (\ref{open}) this step is superfluous.
As the assumption that $H$ is reductive is only used for that step, it thus follows from 
the theorem in \cite{KSS} 
that for each $v\in V$ there exists 
a constant $C>0$ such that 
\begin{equation}  \label{KSS bound} 
|m_{v,\eta}(a\cdot z_0)| \le C p^*(\eta) a^{\Lambda_V}
(1+\|\log a\|)^{d_V} \, \end{equation}
for all $a \in A^+$.

\par The contents of the extended version (\ref{fundbound2}) is that we can replace 
$C$ in (\ref{KSS bound}) by a $G$-continuous norm. 

\par  
For that we just need to add an extra ingredient to the proof in \cite{KSS}. The new ingredient is the 
Casselman comparison theorem (see Remark \ref{comparison} below). 
Let $\Pi \subset \Sigma^+$ the set of simple roots. For a subset $F\subset \Pi$ one 
associates a standard parabolic subalgebra $\pf_F:=\mf_F +\af_F +\nf_F$ with 
$\af_F=\{ X \in \af\mid (\forall \alpha\in F) \alpha(X)=0\}$ etc. 
The comparison theorem asserts in particular that $\nf_F V^\infty$ is closed in $V^\infty$. 
Let $X_1, \ldots, X_n$ be a basis of $ \nf_F$ and consider the surjective map 
of Fr\'echet spaces 
$$\mathcal{T}:\quad\nf_F \otimes V^\infty \to \nf_F V^\infty, \ \ \sum_{j=1}^n X_j \otimes v_j\mapsto 
\sum_{j=1}^n X_j\cdot v_j.$$
The open mapping theorem implies that for every neighborhood $U$
of $0$ in $\nf_F \otimes V^\infty$ there exists a neighborhood $\tilde U$
of $0$ in $\nf_F V^\infty$ such that $\tilde U\subset \mathcal{T}(U)$. Hence
for every $G$-continuous norm $q$ on $V^\infty$ 
there exists a $G$-continuous norm $\tilde q$ on $V^\infty$ 
such that for all $w\in \nf V^\infty$ there exist $v_1, \ldots, v_n\in V^\infty$ 
with 
$w=\mathcal{T}(\sum_{j=1}^n X_j\otimes v_j)$ 
and $q(v_j) \leq \tilde q (w)$. 
Having said that the proof is a simple modification of \cite{KSS}: one obtains a quantitative version of the 
key-step leading to (3.11) in the proof of Theorem 3.2 in \cite{KSS}.  
\end{proof}

In view of Remarks \ref{global bound} and \ref{upperbound with x}, the polar decomposition
(\ref{POLAR}) would allow us to obtain a global bound from the upper bound 
(\ref{fundbound}), if it were on $A_Z^-$ and 
not on the potentially smaller set  $A^- A_H/A_H$. We recall that
spherical spaces with the property
\begin{equation*} %\label{WF}   
A_Z^- = A^-A_H/A_H\end{equation*} 
are called {\it wavefront} (see \cite{KKSS}, Sect. 6).
All symmetric spaces are wavefront. 

\begin{rmk}\label{comparison}  For a Harish-Chandra module $V$ for $(\gf, K)$  and $F\subset \Pi$ we obtain  
an induced Harish-Chandra module $V/\nf_F V=H_0(V, \nf V)$ for the pair $(\gf_F, K_F)$  where $\gf_F=\mf_F +\af_F$. Likewise 
all higher homology groups $H_p(V, \nf_F)$ are Harish-Chandra modules for $(\gf_F, K_F)$. 

The Casselman comparison theorem states that the $\nf_F$-homology groups 
$H_p (V^\infty , \nf_F)$ are separated (Hausdorff) and that the natural mappings 
$H_p(V,  \nf_F)\to H_p (V^\infty , \nf_F)$ induce 
isomorphisms $H_p(V,  \nf_F)^\infty \simeq H_p (V^\infty,   \nf_F)$ 
for all $p\geq 0$. In particular, for $p=0$, we obtain that $\nf_F V^\infty$ is 
closed in $V^\infty$. Up to present, this theorem remains unpublished, although it 
was applied  quite often (see for instance \cite{BO}, Th. 1.5).
Notice that the closedness
of $\nf_F V^\infty$ was crucial in the above proof.

\par An analytic version of the comparison theorem has been established (see \cite{Br}, Thm.~1). It implies
that all $H_p ( V^\omega, \nf_F)$ are separated and  
$$H_p(V, \nf_F)^\omega \simeq H_p ( V^\omega, \nf_F)\qquad (p\geq 0)\, .$$
Here $V^\omega$ is the space of analytic vectors of $V$. 
In fact for the purposes of this paper the analytic comparison theorem is sufficient. 
Eventually it leads to a slight reformulation of the bounds in Theorem \ref{upper bound}
in terms of analytic norms.  We now describe the details. 

\par To begin with we briefly recall the nature 
of the topological vector space $V^\omega$, see \cite{GKS}. We fix a $G$-continuous norm $p$ on $V$ 
and let $V_p$ be the Banach completion of $(V, p)$.  Let $U$ be a bounded open neighborhood 
of $0$ in $\gf_\C$ such that $\exp|_U$ is diffeomorphic onto its image in $G_\C$. For $\e>0$ 
we set $\U_\e:=\exp(U_\e)$ and let $\oline\U_\e$ be its closure.  We denote by $V_p^\e\subset V_p$ 
the subspace of those vectors $v\in V_p$ for which the orbit map $G\to V_p, \ g\mapsto g\cdot v$ extends 
to a continuous map from $G\oline\U_e$ to $E$ which is holomorphic when restricted on $\U_\e$. 
Then $V_p^\e$ becomes a Banach space with norm $p_\e(v) =\max_{g\in \U_\e}  p(g\cdot v)$. 
Note that there are natural continuous inclusions $V_p^\e \to V_p^{\e'}$ for $\e'<\e$. 
Then $V^\omega=\lim_{\e\to 0} V_p^\e$ is the inductive limit of Banach spaces $V_p^\e$. As a topological 
vector space $V^\omega$ is of type DNF (dual nuclear Fr\'echet). 

\par Our concern is now the continuous surjective map 
$$\psi:  \nf_F\otimes V^\omega \to \nf_F V^\omega\,. $$
As closed subspaces of DNF-spaces are DNF we infer from the open mapping theorem 
(which holds true for DNF-spaces, see \cite{HT90}, Appendix A.6) 
that $\psi$ is an open mapping. Let ${\mathcal K}:=\ker \psi$ 
and 
$$\phi: \nf_F V^\omega \to (\nf_F \otimes V^\omega)/ {\mathcal K}$$
the inverse map induced from $\psi$. We are interested in a quantitative description 
of the continuity of $\phi$. Fix $\e>0$. As $V_p^\e \to V^\omega$ is continuous, we get that 
$V_p^\e\cap \nf_FV^\omega$ is closed in the Banach space $V_p^\e$ and further a continuous map 
$$\phi_\e:   V_p^\e \cap \nf_F V^\omega \to (\nf_F \otimes V^\omega)/ {\mathcal K}\, .$$ 
For $\delta>0$ we set ${\mathcal K}_\delta:= (\nf_F\otimes V_p^\delta) \cap {\mathcal K}$.
Then ${\mathcal K}_\delta$ is a closed subspace of the Banach space $\nf_F\otimes V_p^\delta$
and we have continuous inclusions $u_\delta: (\nf_F\otimes V_p^\delta)/ {\mathcal K}_\delta \to 
(\nf_F \otimes V^\omega)/ {\mathcal K}$ with 
$$(\nf_F \otimes V^\omega)/ {\mathcal K}
=\bigcup_{\delta>0} \operatorname{im} \ u_\delta\, .$$ 
Hence the Grothendieck factorization theorem  (see \cite{Gro}, Ch. 4, Sect. 5, Th. 1) applies to the map $\phi_\e$ 
and we obtain an $\e'>0$ such that $\operatorname{im}\phi_\e\subset \im u_{\e'}$ and that the induced map 
between Banach spaces $V_p^\e \cap \nf_F V^\omega \to (\nf_F \otimes V_p^\delta)/{\mathcal K}_\delta$
is continuous.  
In particular, For all $\e>0$ there exists an $\e'>0$ and a constant $C_\e>0$ such that for all $v\in \nf_F V$ there 
exists a presentation $v = \sum_\alpha X_{-\alpha} u_\alpha$  with 
\begin{equation} \label{p to q 2} 
\sum_{\alpha} p_\e(u_\alpha) \leq C_\e p_{\e'} (v)\, .\end{equation}
This was the crucial topological ingedient to the proof of Theorem \ref {upper bound}. The upshot is that we arrive at 
the same bound as in (\ref{fundbound}) but with $q$ replaced by $p_\e$.  For the current application
these bounds are sufficient.  
The reason why we formulated matters in the smooth category is mainly that we consider smooth completions as 
more natural than analytic ones. 
\end{rmk}

\begin{rmk} 
It is possible to obtain a bound on all of $A_Z^-$ but not solely
with the ODE-techniques used in 
this  approach. We will return to that topic in \cite{KKS2}. 
\end{rmk}

\section{Property \propertyUI}

We assume for the moment just that $G$ is a real reductive group and
$Z=G/H$ a unimodular homogeneous space. 
We introduce an integrability condition  for matrix
coefficients on $Z$. It has some similarity with
Kazhdan's property (T).

\par We denote by $\hat G$ the unitary dual of $G$.

\begin{definition}\label{dui} We say that 
$Z=G/H$ has  {\rm property \propertyUI}
provided for all $\pi\in \hat G$
and $\eta\in (\Hc_\pi^{-\infty})^H$ the stabilizer $H_\eta$
of $\eta$ is such that $Z_\eta:=G/H_\eta$ is unimodular 
and there
exists $1\leq p<\infty$ such that
\begin{equation}\label{Lp}
m_{v,\eta}\in L^p(G/H_\eta),
\end{equation}
for all $v\in \Hc_\pi^\infty$.
\end{definition}

The following lemma
shows that it suffices to have (\ref{Lp}) for $K$-finite vectors $v$
of any given type which occurs in $\pi$.

\begin{lemma} \label{lequi} Let $(\pi, \Hc_\pi)$
be irreducible unitary, and let
$\eta\in (\Hc_\pi^{-\infty})^H$ and $1\leq p<\infty$.
The following statements are equivalent:
\begin{enumerate}
\item $m_{v, \eta}\in L^p(Z)$ for all $v\in \Hc_\pi^\infty$.
\item $m_{v, \eta}\in L^p(Z)$ for all $K$-finite vectors in $\Hc_\pi$.
\item $m_{v,\eta}\in L^p(Z)$ for some $K$-finite vector $v\neq 0$.
\end{enumerate}
\end{lemma}

\begin{proof}
Let $V$ be the Harish-Chandra module of $(\pi, \Hc_\pi)$,
i.e. the space of $K$-finite vectors. According to Harish-Chandra,
$V$ is an irreducible $(\gf, K)$-module. The map $v\mapsto
m_{v,\eta}$ is equivariant $V\to C^\infty(Z)$.

We first establish ``$(3)\Rightarrow (2)$''.
Let $v\in V$ be non-zero with $m_{v,\eta}\in L^p(Z)$,
then $v$ generates $V$, and (2) is equivalent with
the statement that $m_{v,\eta}\in L^p(Z)^\infty$.

Let $E$ be the
closed $G$-invariant subspace
of $L^p(Z)$ generated by $m_{v,\eta}$.
As the left action on $L^p(Z)$ is a Banach representation,
the same holds for $E$. The Casimir element
$\cc$ acts by a scalar on $V$, hence it
acts (in the distribution sense) on $E$
by the same scalar. It follows that all $K$-finite
vectors in $E$ are smooth for the
Laplacian $\Delta$ associated to $\cc$.
Thus any $K$-finite vector of $E$ belongs to
$E^\infty\subset L^p(Z)^\infty$ by \cite{BK}, Prop.~3.5.

Finally ``$(2)\Rightarrow (1)$'' follows from
the Casselman-Wallach globalization theorem
(see \cite{BK}),
which implies that the map  $v\mapsto m_{v,\eta}$, $V\to L^p(Z)$,
extends to
$\Hc_\pi^\infty\to L^p(Z)^\infty$ .
\end{proof}

In the definition of property \propertyUI\ we have to take into account that
the stabilizer $H_\eta$ inflates $H$. To discuss this efficiently
it is useful to have an appropriate notion of
factorization for $Z$.

\subsection{Factorization}\label{factorization}
If
there exists a closed subgroup 
$H\subset H^\star\subset G$ 
with $Z^\star=G/H^\star$ unimodular, 
then we call $Z$ {\it factorizable}, 
and call
$$Z\mapsto Z^\star, \ \ gH\mapsto gH^\star$$
a {\it factorization} of $Z$. We call the factorization {\it proper} if
$\dim H<\dim H^\star< \dim G$ and {\it co-compact} if $H^\star/H$ is compact.

\begin{ex}\label{ex7.5}
1) Irreducible symmetric spaces do not have proper factorizations.
In fact, if $(\gf,\hf)$ is a symmetric pair and
$\hf\subsetneq \hf^\star\subsetneq \gf$ a factorization,
then $\gf=\gf_1\times\gf_2$, $\hf=\hf_1\times\hf_2$
and $\hf^\star=\hf_1\times\gf_2$
for some decomposition of $\gf$.
\par\noindent 2) Let $G^n:=G\times \dots \times G$ denote the
direct product of $n$ copies of $G$. Suppose that $G$ is simple.
The homogeneous space
$Z=G^{n}/\diag(G)$ with the diagonal subgroup
$\diag(G)=\{(g,\dots,g)|g\in G\}$
is factorizable if and only if $n>2$.
For example, for $n=3$, we can take
$H^\star=\{(g_1, g_2, g_2)\mid g_1, g_2\in G\} \simeq G\times G$
and permutations thereof.
\par\noindent 3) Let $G/H=\SO(8,\C)/G_2(\C)$. This space is spherical but not wavefront.
The symmetric space $G/H^\star=
\SO(8,\C)/\SO(7,\C)$ is a factorization.
\end{ex}

Note that one has $H_\eta\supset H$ for every $\eta\in (V^{-\infty})^H$. However it is a
priori not clear that $G/H_\eta$ is unimodular if $G/H$ was unimodular. 
Here is an example:

\begin{ex}\label{remIstar} Suppose that $Z=G/N$ with $N$ the unipotent radical 
of a minimal parabolic $P=MAN$. Assume in addition that $G$ is simple 
and that $\dim A \geq 2$. We let $V$ be a generic  
irreducible $K$-spherical unitary principal series. Then $\dim (V^{-\infty})^N=
\#{\mathcal W}$ with ${\mathcal W}$ the Weyl group of the pair $(\gf, \af)$. 
As $V$ was supposed to be generic, the action of $A$ on  $(V^{-\infty})^N$
is semi-simple. Let $\eta\in (V^{-\infty})^N$ such that $a\cdot \eta =
\chi(a) \eta$ for $a\in A$ and a character $\chi: A\to \C^*$. 
If we denote by $A_\chi:=\ker \chi$, then 
$H_\eta= A_\chi N$. As $G$ is simple and $\dim A\geq 2$ it follows that 
$A_\chi$ acts non-trivially on $N$. We may assume that $\chi$ is generic 
enough (not a multiple of $\rho$) so that $A_\chi$ acts in a non-unimodular 
fashion on $N$.  Thus $Z_\eta$ is not unimodular. 
\end{ex}

In case $Z$ is a wavefront real spherical space with $H$ reductive, then the 
factorizations which are of type $G/H_\eta$ will turn out to be
of a special simple shape. 
Let 
\begin{equation}\label{simple ideals dec}
\gf=\gf_1 \times \ldots \times \gf_k
\end{equation}
be a decomposition into 
simple ideals and one dimensional ideals. 
Here the simple ideals $\gf_i$ are unique (up to the order of terms).
For a subset $I \subset \{ 1, \ldots, k\}$ we set 
$$ \hf_I := \hf + \bigoplus_{j \in I} \gf_j\, .$$
Further we set $H_I= H \prod_{j\in I} G_j$ with $G_j\triangleleft G$ the connected 
normal subgroup of $G$ with Lie algebra $\gf_j$.
If $H$ is reductive and $H^\star<G$ is a subgroup with Lie algebra $\hf^\star=\hf_I$
for some $I\subset \{ 1, \ldots, k\}$ and some decomposition (\ref{simple ideals dec}), 
then we call 
$Z^\star:=G/H^\star$ a {\it basic factorization}. 
Note that $H$ reductive implies $H^\star$ reductive in this case.

\begin{ex}
Suppose that $Z=G\times G \times G / \diag (G)$ with
$G$ simple.
Irreducible unitary representations of $G\times G\times G$
are tensor products $\pi=\pi_1\otimes \pi_2 \otimes \pi_3$. If $\pi$ is
non-trivial and $H$-spherical and if one constituent,
say $\pi_1$, is trivial,  then
$H_\eta=\{ (g_1, g_2, g_2)\mid g_1, g_2\in G\}$
is basic.
\end{ex}

In the sequel we need a weaker notion 
of basic factorization. We call a factorization $G/H^\star$ of $G/H$
{\it weakly basic} if 
it is built by a sequence of factorizations
each which is either co-compact or basic.

\subsection{Main theorem on property \propertyUI}
In this last section we resume the assumption of Section \ref{section bounds},
that is, $H$ is a closed subgroup of $G$
with algebraic Lie algebra and with finitely many components.
We call $Z=G/H$ real spherical when $PH$ is open for some minimal parabolic
subgroup. If this is the case, then we say that $G/H$ is wavefront
if its factorization with the Zariski closure of $H$ is wavefront.

\par We can now state one of the main results of the paper.

\begin{theorem}\label{BAN} Let $Z=G/H$ be a wavefront real spherical space with 
$H$ reductive.  Then $Z$ has property \propertyUI{}. Moreover, $Z_\eta=G/H_\eta$ is a weakly basic factorization of $Z$, 
for every unitarizable Harish-Chandra module $V$
and $\eta\in (V^{-\infty})^H$. 
\end{theorem}

\begin{proof} Note that any basic factorization $Z^\star$ of $Z$
will satisfy the same assumptions as requested for $Z$. 
Proceeding by induction, we may thus assume that
all proper basic factorizations $Z^\star$ have property \propertyUI.
We may also assume that $\hf$ contains no
non-trivial ideal of $\gf$ and that $\gf$ is semi-simple.

\par Let $\pi\in \hat G$ be non-trivial and let
$f=m_{v,\eta}$ with non-zero vectors $\eta\in (\Hc_\pi^{-\infty})^H$
and $v\in \Hc_\pi^\infty$.
In view of Lemma \ref{lequi}
it is sufficient for property \propertyUI{}
that $G/H_\eta$ is unimodular and that
$f\in L^p(G/H_\eta)$ for some $1\leq p<\infty$.

\par As $Z$ is wavefront we have 
$Z=\Omega A^- F \cdot z_0$. Note that we may assume that 
$\Omega \subset B$ where $B\subset G$ is a ball. Furthermore, we may assume that
$G$ is semisimple.
 
\par Assume first that $\pi$ is non-trivial on every non-trivial connected
normal subgroup of $G$.
Then by Remark \ref{negative}
the exponent $\Lambda_V$ of $\pi$ is 
contained in ${\rm Int} (\af^+)^*$, i.e. if $\|\cdot\|$ is a norm on 
$\af$, then there exists a constant $C>0$ such that 
$$ - \Lambda_V(X) \geq C\|X\| \qquad (X\in \af^-)\, .$$

\par We apply (\ref{integral as sum}) to the integral of $|f|^p$. 
For each $\gamma\in\Gamma_Z^-$ and $x\in F$ we find
from Theorem \ref{upper bound} and Remark \ref{upperbound with x}
$$|f(g\gamma x\cdot z_0)|\le q(g^{-1}\cdot v)p^*(x\cdot\eta)
\gamma^{\Lambda}(1+\|\log( \gamma)\|)^d$$
and hence by (\ref{bound}) and Remark \ref{global bound}
$$\int_{B\gamma x\cdot z_0} |f(z)|^p\,dz
\le C
\gamma^{p\Lambda_V-2\rho_{\uf}}(1+\|\log( \gamma)\|)^{pd} $$
for some constant $C>0$.
As $\Lambda_V \in  {\rm Int} (\af^+)^*$, this 
can be summed over $\Gamma_Z^-$  
for $p$ sufficiently large, and then
$f\in L^p(Z)$.

We claim that $H_\eta/H$ is compact. Otherwise we find a sequence 
$\omega_n a_n x\cdot z_0\in H_\eta/H$ with $\omega_n \in \Omega$ 
and $a_n\to \infty $ in $A_Z^-$. But this contradicts the bound 
from Theorem \ref{upper bound}. Hence $Z_\eta$ is a co-compact factorization. 
In particular, it is then unimodular and weakly basic. This establishes 
the properties requested in Definition \ref{dui} for this case.
\par On the other hand, assume 
$G$ is not simple and $\pi$ is $1$ on some non-trivial normal subgroup 
$S\triangleleft G$. This subgroup is contained in $H_\eta$ but not in $H$,
hence $Z\to G/SH$ is a proper basic factorization. 
It follows from our inductive hypothesis that $G/SH$ has property \propertyUI{}
and that $G/(SH)_\eta$ is weakly basic.
Since $(SH)_\eta=H_\eta$ we obtain the assertions in the theorem
also for this case.
\end{proof}

This result also has a geometric converse. 

\begin{proposition} Let $Z$ be a wavefront real spherical space with $H$ reductive.  
Then every factorization is weakly basic.
\end{proposition} 
 
\begin{proof} Let $Z^\star$ be a factorization of $Z$. Theorem \ref{BAN} implies that it is 
sufficient to exhibit an irreducible unitary representation $(\pi, {\mathcal H})$ of $G$ 
with $\eta \in ({\mathcal H}^{-\infty})^{H^\star}\subset ({\mathcal H}^{-\infty})^{H}$ such that $H_\eta/ H^\star$ is compact.  
As $Z^\star$ is unimodular, the left regular representation of $G$ on the Hilbert space 
$L^2(Z^\star)$ is unitary. 
Every generic irreducible unitary representation $(\pi, {\mathcal H})$ which is weakly contained 
in $L^2(Z^\star)$ has the requested property. 
\end{proof}

\begin{rmk} The geometric assumption that $H$ is reductive is in some sense
natural in the context of harmonic analysis. In \cite{KK} it is 
shown that if $Z=G/H$ is unimodular and $H$ is self-normalizing, 
then $H$ is reductive. 
\end{rmk}

\end{document}